\documentclass[11pt]{article}

\usepackage[a4paper,left=1.15in,right=1.15in,top=1.1in,bottom=1.15in]{geometry}
\usepackage{amsmath,amssymb,amsthm,mathtools}
\usepackage{hyperref}

\newtheorem{theorem}{Theorem}
\newtheorem{corollary}{Corollary}
\newtheorem{remark}{Remark}

\newcommand{\T}{\mathbb T}
\newcommand{\D}{\mathbb D}
\newcommand{\C}{\mathbb C}
\newcommand{\dm}{\,dm}
\newcommand{\norm}[1]{\|#1\|}
\newcommand{\Cat}{\mathsf G}

\title{Sharp Estimates for Conjugate Functions with Applications to Trigonometric Polynomials}
\author{Silouanos Brazitikos}

\date{}

\begin{document}

\maketitle
\begin{abstract}
We prove a sharp estimate for conjugate functions using a harmonic
majorant in a half-strip. As an application, we remove the logarithmic loss from
a theorem of Papadopoulos on minima of trigonometric polynomials and obtain the optimal order in the corresponding inequality.
\end{abstract}
\section{Introduction}

Let
\(
        0<n_1<\cdots<n_N
\)
be distinct positive integers and let \(a_1,\dots,a_N\in\C\). Let \(\T=\mathbb R/2\pi\mathbb Z\) equipped with normalized Haar measure
\(dm(x)=dx/(2\pi)\). For $x\in\T$ we write
\[
        P(x)=\sum_{k=1}^N a_k e^{in_kx}
        =f(x)+i\widetilde f(x),
\]
and set
\[
        M=-\min_{x\in\T} f(x).
\]
The special case \(a_1=\cdots=a_N=1\) is closely related to Chowla's cosine
problem. Namely, if $A$ is a $n$-element set and 
\[
        C_A(x):=\sum_{a\in A}\cos(ax),
        \qquad
        S_A(x):=\sum_{a\in A}\sin(ax),
\]
then \(C_A\) and \(S_A\) are conjugate trigonometric sums. Chowla's conjecture
asserts that
\[
        -\min_x C_A(x)\geq c\sqrt N
\]
with an absolute constant \(c>0\), uniformly over all sets \(A\) of \(N\) distinct
positive integers. Very recently, the first polynomial lower bounds for Chowla's cosine problem
were obtained independently by Bedert \cite{BedertChowla} and by
Jin, Milojevi\'c, Tomon and Zhang \cite{JinMilojevicTomonZhang}.
Bedert proved that
\[
        \min_x \sum_{a\in A}\cos(ax)
        \leq
        -|A|^{1/7-o(1)},
\]
while Jin, Milojevi\'c, Tomon and Zhang obtained
\[
        \min_x \sum_{a\in A}\cos(ax)
        \leq
        -|A|^{1/10-o(1)}.
\]
At the time of Papadopoulos' work \cite{Papadopoulos}, the strongest general lower
bound for this cosine problem was due to Bourgain \cite{Bourgain} and was super-logarithmic but
not polynomial in \(N\). A second, related question
concerns the smallest possible size of the conjugate sum
\(
        \norm{S_A}_\infty.
\)
The trivial lower bound is of order \(\sqrt N\), while Bourgain constructed
examples with
\[
        \norm{S_A}_\infty\leq C N^{2/3}
\]
as reported in Kahane's book \cite{Kahane}. In the opposite direction, Konyagin \cite{Konyagin}
proved that every such sine sum satisfies
\[
        \norm{S_A}_\infty
        \geq
        c\sqrt{\frac{N\log N}{\log\log N}}.
\]
Papadopoulos observed that these two questions are quantitatively connected.
His Theorem A states that if \(|a_k|\geq1\) for all \(k\), then
\[
        M
        \geq
        c\,\frac{N}{\norm{\widetilde f}_\infty\log N}.
\]
Thus an upper bound on the conjugate function \(\widetilde f\) forces a lower
bound on the negative minimum of \(f\). In particular, in the cosine problem,
flatness of the sine sum implies that the cosine sum must take a substantially
negative value \cite{Papadopoulos}. Papadopoulos also proved a related estimate
for the mixed sums
\[
        \sum_{k=1}^N(\cos n_kx+\sin n_kx),
\]
showing that their negative minimum is at least of order \(\sqrt N/\log N\);
this improved earlier logarithmic estimates attributed to Pichugov and discussed
by Belov and Konyagin \cite{BelovKonyagin,Papadopoulos}.
The purpose of the present note is an elementary improvement of
Papadopoulos' first estimate. This follows from a sharp inequality for
conjugate functions, obtained by solving a half-strip extremal problem. The underlying estimate is not specific to trigonometric
polynomials. It is a sharp inequality for conjugate functions. Recall
that the \(L^p\)-boundedness of the conjugate function on the circle goes back to
M. Riesz \cite{Riesz1924,Riesz1928}; the weak type \((1,1)\) theorem is due to
Kolmogorov \cite{Kolmogorov}; the sharp \(L^p\) constants for the Hilbert
transform and conjugate function were obtained by Pichorides, following earlier
work of Gohberg and Krupnik in special cases \cite{GohbergKrupnik,PichoridesConstants};
and the sharp weak type \((1,1)\) constant was found by Davis, with a later
alternative proof by Baernstein \cite{Davis,Baernstein}. We shall use a different
but related sharp principle: if \(u\geq0\) and \(v=Hu\) is bounded, then the
analytic function \(u+iv\) maps the disk into a half-strip. Solving the
corresponding Dirichlet problem in that half-strip yields the sharp estimate
\[
        \int_{\mathbb T}|v|^p\,dm
        \leq
        A_p\norm{v}_\infty^{p-1}\int_{\mathbb T}u\,dm,
        \qquad p>1,
\]
where
\[
        A_p
        =
        p\int_0^1\frac{t^{p-1}}{\sin(\pi t/2)}\,dt.
\]
This is a sharp \(L^p\) inequality for conjugate functions.

Applying this result with
\[
        u=M+f,
        \qquad
        v=\widetilde f,
\]
we obtain, for every \(p>1\),
\[
        M
        \geq
        \frac{\norm{\widetilde f}_p^p}
        {A_p\norm{\widetilde f}_\infty^{p-1}}.
\]
The case \(p=2\) is especially relevant for Papadopoulos' theorem, since
orthogonality gives
\[
        \norm{\widetilde f}_2^2
        =
        \frac12\sum_{k=1}^N |a_k|^2.
\]
Moreover,
\[
        A_2
        =
        2\int_0^1\frac{t}{\sin(\pi t/2)}\,dt
        =
        \frac{16\mathsf G}{\pi^2},
\]
where
\[
        \mathsf G
        =
        \sum_{j=0}^{\infty}\frac{(-1)^j}{(2j+1)^2}
\]
is Catalan's constant. Hence
\[
        M
        \geq
        \frac{\pi^2}{32\mathsf G}
        \frac{\sum_{k=1}^N |a_k|^2}
        {\norm{\widetilde f}_\infty}.
\]
In particular, if \(|a_k|\geq1\) for all \(k\), then
\[
        M
        \geq
        \frac{\pi^2}{32\mathsf G}
        \frac{N}{\norm{\widetilde f}_\infty}.
\]
This removes the logarithmic loss from Papadopoulos' Theorem A.

\section{A sharp \(L^p\) inequality}

\begin{theorem}
Let \(p>1\). Let
\(
        F=u+iv
\)
be analytic in the unit disk \(\D\), and suppose that \[
        u\geq0,\qquad |v|\leq B
\]
almost everywhere on \(\T\). Assume in addition that
\[
        \int_{\T}v\,dm=0.
\]
Then
\[
        \int_{\T}|v|^p\,dm
        \leq
        A_p B^{p-1}\int_{\T}u\,dm,
\]
where
\[
     A_p=
        p\int_0^1
        \frac{t^{p-1}}{\sin(\pi t/2)}\,dt.
\]
The constant \(A_p\) is best possible.
\end{theorem}

\begin{proof}
Due to homogeneity, it is enough to prove the result when \(B=1\). Then
\(
        F(\D)\subset \Omega\), where
\[
        \Omega=\{x+iy:x>0,\ |y|<1\}.
\]
Let \(U_p\) be the harmonic function in \(\Omega\) with boundary values
\[
        U_p(0,y)=|y|^p,\qquad -1<y<1,
\]
and
\[
        U_p(x,1)=U_p(x,-1)=1,\qquad x>0.
\]
We first claim that for all $(x,y)\in \Omega$
\[
        |y|^p\leq U_p(x,y).
\]
Indeed, \(|y|^p-U_p(x,y)\) is subharmonic in the distributional sense, since
\(|y|^p\) is convex for \(p>1\), and it is non-positive on the boundary of
\(\Omega\). The maximum principle on truncated half-strips, followed by passage
to the limit, gives the claim.

Applying the claim to \((x,y)=(u(z),v(z))\), we get
\[
        |v(z)|^p\leq U_p(u(z),v(z)).
\]
Since \(U_p\) is
harmonic in the half-strip \(\Omega\) and \(F=u+iv\) is analytic with values in
\(\Omega\), the composition
\[
        z\mapsto U_p(u(z),v(z))=U_p(F(z))
\]
is harmonic in \(\D\). Hence, by the mean value property for harmonic functions,
for \(0<r<1\),
\[
        \int_{\T}
        U_p(u(re^{it}),v(re^{it}))\,dm(t)
        =
        U_p(u(0),v(0)).
\]
Letting \(r\to1^-\), and using the radial boundary values, we obtain
\[
        \int_{\T}U_p(u,v)\,dm
        =
        U_p(u(0),v(0)).
\]

On the other hand, since \(u\) and \(v\) are the real and imaginary parts of the
analytic function \(F\), their boundary means are the real and imaginary parts of
\(F(0)\). Equivalently,
\[
        F(0)
        =
        \int_{\T}F\,dm
        =
        \int_{\T}u\,dm
        +
        i\int_{\T}v\,dm.
\]
By our hypothesis, \(
        \int_{\T}v\,dm=0,
\)
therefore \(F(0)\) is real and positive, and
\[
        F(0)=u(0)=\int_{\T}u\,dm:=X.
\]
Then
\[
        U_p(u(0),v(0))=U_p(X,0).
\]
Combining this with
\[
        |v|^p\leq U_p(u,v)
\]
gives
\[
        \int_{\T}|v|^p\,dm
        \leq
        \int_{\T}U_p(u,v)\,dm
        =
        U_p(X,0).
\]
It remains to estimate \(U_p(X,0)\). Put
\[
        W_p(x,y)=1-U_p(x,y).
\]
Then \(W_p\) is harmonic in the half-strip
\[
        \Omega=\{(x,y):x>0,\ -1<y<1\},
\]
and satisfies
\[
        W_p(x,\pm1)=0,
        \qquad
        W_p(0,y)=1-|y|^p.
\]

We use the Dirichlet eigenfunctions of the interval \((-1,1)\). Since the
boundary datum \(1-|y|^p\) is even, we expand it in the even Dirichlet
eigenfunctions. These are
\[
        \cos(\lambda_jy),
        \qquad
        \lambda_j=\left(j+\frac12\right)\pi,
        \qquad j=0,1,2,\dots,
\]
because \(\cos(\lambda_j)=0\). They form an orthogonal basis of the even
subspace of \(L^2(-1,1)\), and
\[
        \int_{-1}^{1}\cos^2(\lambda_jy)\,dy=1.
\]
Thus
\[
        1-|y|^p
        =
        \sum_{j=0}^{\infty}c_j^{(p)}\cos(\lambda_jy),
\]
where
\[
        c_j^{(p)}
        =
        \int_{-1}^{1}(1-|y|^p)\cos(\lambda_jy)\,dy
        =
        2\int_0^1(1-y^p)\cos(\lambda_jy)\,dy.
\]
For each \(j\), the function
\[
        e^{-\lambda_jx}\cos(\lambda_jy)
\]
is harmonic in the half-strip and vanishes on \(y=\pm1\). Hence the series
\[
        \widetilde W_p(x,y)
        :=
        \sum_{j=0}^{\infty}
        c_j^{(p)}e^{-\lambda_jx}\cos(\lambda_jy)
\]
defines, for \(x>0\), the bounded harmonic extension of the datum
\(1-|y|^p\). By uniqueness of the bounded solution of the Dirichlet problem in
the half-strip,
\[
        W_p(x,y)=\widetilde W_p(x,y).
\]
Consequently,
\[
        W_p(x,y)
        =
        \sum_{j=0}^{\infty}
        c_j^{(p)}e^{-\lambda_jx}\cos(\lambda_jy).
\]

Define
\[
        g_p(x):=U_p(x,0).
\]
Since \(W_p=1-U_p\), we have
\[
        g_p(x)
        =
        1-W_p(x,0)
        =
        1-
        \sum_{j=0}^{\infty}
        c_j^{(p)}e^{-\lambda_jx}.
\]
In particular \(g_p(0)=U_p(0,0)=0\).

For \(x>0\), the exponential factor \(e^{-\lambda_jx}\) justifies termwise
differentiation. Indeed, by integration by parts,
\[
\begin{aligned}
        c_j^{(p)}\lambda_j
        &=
        2\int_0^1(1-y^p)\lambda_j\cos(\lambda_jy)\,dy  \\
        &=
        2\int_0^1(1-y^p)\,d(\sin(\lambda_jy))  \\
        &=
        2p\int_0^1 y^{p-1}\sin(\lambda_jy)\,dy.
\end{aligned}
\]
The boundary term vanishes because \(1-y^p=0\) at \(y=1\) and
\(\sin(\lambda_jy)=0\) at \(y=0\). Hence
\[
        |c_j^{(p)}\lambda_j|
        \leq
        2p\int_0^1y^{p-1}\,dy
        =
        2.
\]
Therefore, for every \(x>0\),
\[
        \sum_{j=0}^{\infty}
        |c_j^{(p)}|\lambda_j e^{-\lambda_jx}
        <
        \infty,
\]
and so
\[
        g_p'(x)
        =
        \sum_{j=0}^{\infty}
        c_j^{(p)}\lambda_j e^{-\lambda_jx}.
\]
Using the formula for \(c_j^{(p)}\lambda_j\), we obtain
\[
        g_p'(x)
        =
        2p\sum_{j=0}^{\infty}e^{-\lambda_jx}
        \int_0^1 y^{p-1}\sin(\lambda_jy)\,dy.
\]
Since \(x>0\), the series is absolutely convergent, and Fubini's theorem gives
\[
        g_p'(x)
        =
        2p\int_0^1 y^{p-1}
        \left(
        \sum_{j=0}^{\infty}e^{-\lambda_jx}\sin(\lambda_jy)
        \right)dy.
\]
For \(x>0\) and \(0<y<1\), set
\[
        K_x(y):=
        \sum_{j=0}^{\infty}
        e^{-\lambda_jx}\sin(\lambda_jy).
\]
Thus
\[
        g_p'(x)
        =
        2p\int_0^1 y^{p-1}K_x(y)\,dy.
\]
We now compute \(K_x\). Since
\[
        \lambda_j=\left(j+\frac12\right)\pi,
\]
writing
\[
        r=e^{-\pi x},
        \qquad
        \theta=\pi y,
\]
we get
\[
        K_x(y)
        =
        \operatorname{Im}
        \left(
        \frac{r^{1/2}e^{i\theta/2}}{1-r e^{i\theta}}
        \right).
\]
Hence
\[
        K_x(y)
        =
        \frac{
        r^{1/2}(1+r)\sin(\theta/2)
        }{
        1-2r\cos\theta+r^2
        }.
\]
Equivalently,
\[
        K_x(y)
        =
        \frac{
        e^{-\pi x/2}(1+e^{-\pi x})\sin(\pi y/2)
        }{
        1-2e^{-\pi x}\cos(\pi y)+e^{-2\pi x}
        }.
\]

We shall use the elementary bound
\[
        K_x(y)\leq \frac{1}{2\sin(\pi y/2)}.
\]
Indeed, putting
\[
        s=\sin(\theta/2),
        \qquad
        a=\sqrt r,
\]
this inequality is equivalent to
\[
        2s\cdot
        \frac{a(1+a^2)s}{(1-a^2)^2+4a^2s^2}
        \leq 1.
\]
After rearrangement, this becomes
\[
        (1-a^2)^2+4a^2s^2-2a(1+a^2)s^2\geq0.
\]
The left-hand side equals
\[
        (1-a)^2\bigl((1+a)^2-2as^2\bigr),
\]
which is non-negative since \(0<s\leq1\).

Therefore, for every \(x>0\),
\[
        g_p'(x)
        =
        2p\int_0^1 y^{p-1}K_x(y)\,dy
        \leq
        p\int_0^1
        \frac{y^{p-1}}{\sin(\pi y/2)}\,dy
        =
        A_p.
\]
Since \(g_p(0)=0\), it follows that
\[
        g_p(X)
        =
        \int_0^X g_p'(x)\,dx
        \leq
        A_pX.
\]

It remains to identify the sharp constant. For fixed \(0<y<1\), the explicit
formula for \(K_x(y)\) gives
\[
        K_x(y)\longrightarrow
        \frac{1}{2\sin(\pi y/2)}
        \qquad (x\to0^+).
\]
Moreover, the bound just proved gives the domination
\[
        y^{p-1}K_x(y)
        \leq
        \frac{y^{p-1}}{2\sin(\pi y/2)}.
\]
The right-hand side is integrable on \((0,1)\), precisely because \(p>1\).
Thus, by dominated convergence,
\[
        \lim_{x\to0^+}g_p'(x)
        =
        2p\int_0^1
        y^{p-1}
        \frac{1}{2\sin(\pi y/2)}\,dy
        =
        A_p.
\]
Consequently,
\[
        \lim_{X\to0^+}\frac{g_p(X)}{X}=A_p.
\]
This gives the sharpness statement needed below.
\end{proof}

\section{Consequences for Papadopoulos' setting}

We now return to the trigonometric polynomial
\[
        P(x)=\sum_{k=1}^N a_k e^{in_kx}
        =
        f(x)+i\widetilde f(x),
\]
and set
\[
        M=-\min_{\T}f.
\]

\begin{corollary}
For every \(p>1\),
\[
         M
        \geq
        \frac{\norm{\widetilde f}_p^p}
        {A_p\norm{\widetilde f}_\infty^{p-1}},
\]
where
\[
        A_p=
        p\int_0^1
        \frac{t^{p-1}}{\sin(\pi t/2)}\,dt.
\]
\end{corollary}

\begin{proof}
Consider
\[
        F(x)=M+P(x)=u(x)+iv(x),
\]
where
\[
        u(x)=M+f(x),
        \qquad
        v(x)=\widetilde f(x).
\]
By definition of \(M\),
\[
        u(x)\geq0.
\]
Moreover,
\[
        \norm{v}_\infty=\norm{\widetilde f}_\infty.
\]
Since \(P\) has no constant term,
\[
        \int_{\T}f\,dm=0,
\]
and therefore
\[
        \int_{\T}u\,dm=M.
\]
Applying the preceding theorem gives
\[
        \norm{\widetilde f}_p^p
        \leq
        A_p\norm{\widetilde f}_\infty^{p-1}M.
\]
Rearranging proves the claim.
\end{proof}

Taking \(p=2\) gives the logarithm-free improvement of Papadopoulos' Theorem A.

\begin{corollary}
One has
\[
        M
        \geq
        \frac{\pi^2}{32\Cat}
        \frac{\sum_{k=1}^N |a_k|^2}
        {\norm{\widetilde f}_\infty}.
\]
In particular, if \(|a_k|\geq1\) for all \(k\), then
\[
        M
        \geq
        \frac{\pi^2}{32\Cat}
        \frac{N}{\norm{\widetilde f}_\infty}.
 \]
\end{corollary}

\begin{proof}
For \(p=2\),
\[
        A_2=
        2\int_0^1\frac{t}{\sin(\pi t/2)}\,dt
        =
        \frac{16\Cat}{\pi^2}.
\]
The \(L^2\) norm of \(\widetilde f\) is computed by orthogonality. If
\(a_k=\alpha_k+i\beta_k\), then
\[
        \mathrm{Im}(a_ke^{in_kx})
        =
        \alpha_k\sin(n_kx)+\beta_k\cos(n_kx).
\]
Since the frequencies \(n_k\) are distinct,
\[
        \norm{\widetilde f}_2^2
        =
        \frac12\sum_{k=1}^N |a_k|^2.
\]
The \(p=2\) case of the preceding corollary gives
\[
        M
        \geq
        \frac{1}{A_2}
        \frac{\norm{\widetilde f}_2^2}{\norm{\widetilde f}_\infty}
        =
        \frac{\pi^2}{16\Cat}
        \frac{\norm{\widetilde f}_2^2}{\norm{\widetilde f}_\infty}.
\]
Substituting the identity for \(\norm{\widetilde f}_2^2\), we obtain
\[
        M
        \geq
        \frac{\pi^2}{32\Cat}
        \frac{\sum_{k=1}^N |a_k|^2}
        {\norm{\widetilde f}_\infty}.
\]
If \(|a_k|\geq1\), then \(\sum |a_k|^2\geq N\), and the final claim follows.
\end{proof}

\begin{remark}
For \(p\neq2\), the corollary gives a sharp moment estimate
\[
        M
        \geq
        \frac{\norm{\widetilde f}_p^p}
        {A_p\norm{\widetilde f}_\infty^{p-1}}.
\]
The case \(p=2\) is distinguished in Papadopoulos' problem because Parseval's
identity converts \(\norm{\widetilde f}_2^2\) exactly into
\(\frac12\sum |a_k|^2\).
\end{remark}
\begin{remark}
The estimate is sharp in order of $N$. Indeed, by a theorem of
Rudin \cite{Rudin1959}, there exists a sequence
\(\varepsilon_n\in\{-1,1\}\), \(n=1,2,\dots\), such that for every \(N\geq1\),
\[
        \left\|
        \sum_{n=1}^N \varepsilon_n z^n
        \right\|_{L^\infty(\mathbb T)}
        \leq
        5\sqrt N .
\]
Set
\[
        P_N(z)=\sum_{n=1}^N \varepsilon_n z^n
\]
and write
\[
        P_N(e^{ix})=f_N(x)+i\widetilde f_N(x).
\]
Then
\[
        M_N:=-\min_{x\in\mathbb T} f_N(x)
        \leq
        \|f_N\|_\infty
        \leq
        \|P_N\|_\infty
        \leq
        5\sqrt N,
\]
and similarly
\[
        \|\widetilde f_N\|_\infty
        \leq
        \|P_N\|_\infty
        \leq
        5\sqrt N.
\]
Therefore
\[
        M_N\|\widetilde f_N\|_\infty
        \leq
        25N.
\]
Thus the inequality
\[
        M\|\widetilde f\|_\infty
        \gtrsim
       N
\]
has the correct order.
\end{remark}

\thanks{\noindent {\bf Keywords:} Conjugate functions; Hilbert transform; trigonometric polynomials; minima of
cosine sums; one-sided inequalities; harmonic majorants; sharp constants;
Chowla cosine problem.}

\smallskip

\thanks{\noindent {\bf 2020 MSC:} Primary: 42A05, 42A50. Secondary: 30H10, 31A05, 42A16.}

\bigskip

\bigskip 

\medskip 

\noindent \textsc{Silouanos \ Brazitikos}: Department of Mathematics \& Applied Mathematics, University of Crete, Voutes Campus, 70013 Heraklion, Greece.

\smallskip

\noindent \textit{E-mail:} \texttt{silouanb@uoc.gr}
\end{document}